\documentclass[11pt, twoside]{article}
\usepackage[english]{babel}

\usepackage{amsmath,amsthm,amscd,latexsym, eucal,epsfig,enumitem,bbold,color}

\usepackage{amssymb}
\usepackage{float} 
\usepackage{graphicx} 
\usepackage{pstricks,pstricks-add,pst-math,pst-xkey}
\usepackage{multicol}
\usepackage[labelformat=simple]{subcaption}

\newtheorem{theorem}{Theorem}
\newtheorem{proposition}[theorem]{Proposition}
\newtheorem{lemma}[theorem]{Lemma}
\newtheorem{corollary}[theorem]{Corollary}

\theoremstyle{definition}
\newtheorem{definition}[theorem]{Definition}

\newtheorem{remark}[theorem]{Remark}
\newtheorem{exempl}[theorem]{Example}
\theoremstyle{theorem}
\newtheorem*{mainResult*}{Main Result}
\newtheorem*{remark*}{Remark}
\theoremstyle{definition}
\newtheorem*{definition*}{Definition}
\DeclareMathAlphabet{\mathpzc}{OT1}{pzc}{m}{it}

\begin{document}

\title{Synchronization in abstract mean field models}

\author{W. Oukil, \\ \quad \\
\small\text{Laboratory of Dynamical Systems (LSD), Faculty of Mathematics,}\\
\small\text{ University of Sciences and Technology Houari Boumediene,}\\
\small\text{ BP 32 El Alia 16111, Bab Ezzouar, Algiers, Algeria.}}

\date{\today}

\maketitle
\begin{abstract}
We show in this paper a  sufficient condition for the existence of solution, the synchronized and the periodic locked state in abstract mean field models.
\end{abstract}




\section{Introduction}\label{Chap2hypothesisnonlinearsynch}
This article is a generalization of the result obtained in \cite{OukilKessiThieullen}. The  class of abstract mean field systems that we study in this article is given by the  next systems. The {\it{periodic not-perturbed }} system
\begin{equation}\label{Chap1NotPerturbedSystem} \tag{PNP}
\dot{x}_i=F(X,x_i),\quad i=1,..,N,\quad t\geq t_0,
\end{equation}
and the {\it{perturbed}} system
\begin{equation}\label{SystemGeneralMeanField} \tag{P}
\dot{x}_i=F(X,x_i)+H_i(X),\quad i=1,..,N,\quad t \geq t_0,
\end{equation}
where $N\geq2$ and $X=(x_1,\ldots,x_N)$ is the state of the system. $F : \mathbb{R}^N\times\mathbb{R}\to\mathbb{R}$ and $H=(H_1,\ldots,H_N): \mathbb{R}^N\to\mathbb{R}^N$ are a  $C^1$ functions. The function $H$ is a  perturbation of the system \eqref{Chap1NotPerturbedSystem}. We note $\Phi^t$ the flow of the system \eqref{SystemGeneralMeanField} (in particular of the system \eqref{Chap1NotPerturbedSystem}). We have take the two  systems because the results seem not trivial for the periodic  not-perturbed system.
\subsection{Notations and definitions}
 In this section, we introduce some notations and definitions. For $q,p\in \mathbb{N}^*$ let $G$ be a function from $\mathbb{R}^q$ to $\mathbb{R}^p$. Put $G=(G_1,\ldots,G_p)$ we consider the quasi-norm on the space of continues functions from $\mathbb{R}^q$ to $\mathbb{R}^p$ defined by the next quantity
\[
||G||_{B}=\sup_{Y\in B}\max_{1\le i \le p}|G_i(Y)|,
\]
\[
\text{where}\ B=\{Y=(y_1,\ldots,y_q)\in\mathbb{R}^q:\ \max|y_i-y_j|\le1\}.
\] 
This quasi-norm is a norm on the space of continues functions from $B$ to $\mathbb{R}^p$. We note $d^iG$, $i=1,2,\ldots$, the $i^{\text{th}}$ differential of $G$.  We define
\[
||dG||_{B}=\max_{\substack{1\le i \le p\\ 1 \le j \le q}}||\partial_j G_i(Y)||_B,\quad||d^2G||_{B}=\max_{\substack{1\le i \le p\\ 1 \le j,k \le q}}||\partial_k\partial_j G_i(Y)||_B.
\]
Let $G:\mathbb{R}^N\times\mathbb{R}\to\mathbb{R}$, $Y=(y_1,\ldots,y_N)\in \mathbb{R}^N$ and $z\in \mathbb{R}$, We note
\begin{equation*}
\partial_i G(Y,z):=\left\{
\begin{aligned}
&\frac{\partial}{\partial z}G(Y,z) \quad\text{if}\quad i = N+1,\\ 
& \frac{\partial}{\partial{y_i}}G(Y,z)\quad\text{if}\quad i\in\{1,\ldots,N\}.
\end{aligned}
\right.
\end{equation*}
A function $G :\mathbb{R}^q\to \mathbb{R}^p$ is called $\mathbb{1}$-{\it{periodic}} in the sense of the following definition
\label{1periodicdefintion}
\begin{definition}\label{diagonalperiodicity}[$\mathbb{1}$-periodic function] Let $G :\mathbb{R}^q\to \mathbb{R}^p$ be a function and note $\mathbb{1}:=(1,\ldots,1)\in \mathbb{R}^q$. The function $G$ is called $\mathbb{1}$-{\it{periodic}} if
\[{G}(Y+\mathbb{1})={G}(Y),\quad\forall Y\in\mathbb{R}^q.\]
\end{definition}
Remark that the previous definition do not imply  that the function $G$ is periodic relative to each variable.
Now we define a positive $\Phi^t$-invariant set,
\begin{definition}
Suppose that the flow $\Phi^{t}$ of system \eqref{SystemGeneralMeanField} exists for every $t\geq t_{0}$. We say that a open set  $C \subset \mathbb{R}^N$ is a positive $\Phi^t$-invariant if $\Phi^{t}(C) \subset C$ for all $t \geq t_{0}$.
\end{definition}
Synchronization and locking may have several meanings or definitions depending on the authors. We choose the following definitions.
\begin{definition}[Dynamical oscillator]\label{DefinitionVitessePositive}
The oscillator $x_i(t)$ of a solution $X(t)=(x_1(t),..,x_N(t))$ of system \eqref{SystemGeneralMeanField} is called {\it{dynamical}} if there exists $t_0\in \mathbb{R}$ such that
\[
\inf_{t\geq t_0}\dot{x}_i(t)>0.
\]
\end{definition}
\begin{definition}[Synchronisation]\label{Chap1DefinitionSynchronisationDispersions}
We say that the oscillators $\{x_i(t)\}_{i=1}^{N}$ are synchronized  if they are dynamical and if $\sup_{1 \leq i,j \leq N}  |x_i(t)-x_j(t)| $ is bounded from above uniformly in time $t\ge t_0$.
\end{definition}
\begin{definition}[Periodic locked solution]\label{Chap1DefinitionAccrochagePeriodique}
We say that the  oscillators $\{x_i(t)\}_{i=1}^{N}$ are periodically locked to the frequency $\rho>0$ if they are synchronized and if there exist a periodic functions $\Psi_i(t)$  such that
\[
x_i(t) = \rho t + \Psi_i(t), \quad \forall i= 1\ldots N,  \ \forall t\geq t_0.
\]
\end{definition}
\subsection{Synchronization Hypothesis $(H)$ and $(H_*)$}
The goal is to prove the existence of the synchronization stateof the syst\`eme \eqref{SystemGeneralMeanField} when $||H||_B\approx 0$. Consider the following hypotheses
\label{hypabriviatHH}
\begin{align*}
(H)&\quad \left\{
    \begin{array}{ll}
	F\ \text{is }\ C^2,\quad\text{and}\quad \max\{||F||_{B},||dF||_{B},||d^2F||_{B}\}<+\infty,\\
	F\ \text{is}\  \mathbb{1}\text{-periodic}\ \text{and}\ \min_{s\in [0,1]}F(s\mathbb{1},s)>0,	   
    \end{array}
\right.\\
(H_*)&\quad\int_{0}^{1}\frac{\partial_{N+1} F(s\mathbb{1},s)}{F(s\mathbb{1},s)}ds <0.
\end{align*}
We call the hypothesis $(H_*)$ the {\it{synchronization hypothesis}}.  The particularity of the hypothesis $(H_*)$ is the fact :  $H\approx0$ and $x_i\approx x_j (\approx x)$ implies that the system \eqref{SystemGeneralMeanField} is equivalent to
\[\frac{d}{dt}{x}_{i}\approx F(x\mathbb{1},x)\ \text{and}\ \frac{d}{dt}(x_i-x_j)\approx \partial_{N+1} F(x\mathbb{1},x)(x_i-x_j).\]
The condition $\min_{s\in [0,1]}F(s\mathbb{1},s)>0$  is a sufficient condition to get a dynamical oscillators as defined in definition \ref{DefinitionVitessePositive}.
\subsection{Main Results}\label{Chap1SectionHypotheseetmainresultsNonlineaire}
The following main result {\bf{I}} shows the existence of the solution and a synchronized state as defined in Definition  \ref{Chap1DefinitionSynchronisationDispersions}
\begin{mainResult*}[{\bf{I}}]
We consider the system \eqref{SystemGeneralMeanField}. Suppose that $F$ satisfies the hypotheses  $(H)$ and $(H_*)$ then there exists $D_* >0$ such that for all $D \in (0,D_*]$ there exists $\mathpzc{r}>0$  and a open set $C_{\mathpzc{r}}$  of the form, 
\[
C_{\mathpzc{r}} := \Big\{ X=(x_i)_{i=1}^N \in \mathbb{R}^N \,:\, \exists \nu \in \mathbb{R}, \quad\max_{i}|x_i-\nu| < \Delta_{\mathpzc{r}}(\nu) \Big\},
\]
where  $\Delta_{\mathpzc{r}} : \mathbb{R} \to (0,D]$ is a $C^1$ and $1$-periodic function, such that for every $C^1$ function   $H$  satisfying $||H||_{B}<\mathpzc{r}$ we have
\begin{enumerate}
\item {\rm{Existence of solution}}. The flow $\Phi^t$ of the system \eqref{SystemGeneralMeanField} exists for all initial condition $X\in C_{\mathpzc{r}}$ and for all  $t\geq t_{0}$.
\item {\rm{Synchronization}}. The open set $C_{\mathpzc{r}}$ is positive $\Phi^t$-invariant . Further, for every  $X\in C_{\mathpzc{r}}$ we have
\[\min_{1\le i\le N}\inf_{t\geq t_0}\frac{d}{dt}\Phi^t_i(X)>0\  \text{and}\ \ |\Phi^{t}_{i}(X)-\Phi^{t}_{j}(X)|<2D, \forall1\le i,j \le N,\ \forall t\geq t_{0}.\]
\end{enumerate}
\end{mainResult*}
The next main result {\bf{II}} shows the existence of a periodic locked solution as defined in definition \ref{Chap1DefinitionAccrochagePeriodique}
\begin{mainResult*}[{\bf{II}}]
We consider the system \eqref{SystemGeneralMeanField}. Suppose that $F$ satisfies the hypotheses  $(H)$ and $(H_*)$ then there exists $D_* >0$ such that for all $D \in (0,D_*]$ there exists $\mathpzc{r}>0$  such that for every $C^1$ and  $\mathbb{1}$-periodic function  $H$ satisfying  $||H||_{B}<\mathpzc{r}$,  there exists an open set $C_{\mathpzc{r}}$ (same in main result {\rm(I)}) and a initial condition $X_*\in C_{\mathpzc{r}}$ such that
\[
\Phi^t_i(X_*) = \rho t + \Psi_{i,X_*}(t), \quad \forall i=1,..,N, \ \forall  t\geq t_0,
\]
where $\rho>0$ and $\Psi_{i,X_*} : \mathbb{R} \to \mathbb{R}$ are a $C^1$ and $\frac{1}{\rho}$-periodic functions.
\end{mainResult*}
\begin{remark*}\ 
The result {\bf{I}} can be generalized to a function $H(t,X)$ which depend on the variable time $t$.
\end{remark*}
\subsection{Remarks and motivation}
The results can be applied to the model of coupled oscillators  as the Winfree   \cite{WinfreeModel} and the Kuramoto model \cite{kuramoto1}.

\begin{exempl}\label{Example:WinfreeKuramoto}[Winfree and Kuramoto Models]
Winfree \cite{WinfreeModel}   proposed a model  describing the synchronization of a population of organisms or {\it oscillators} that interact simultaneously.  The Winfree model is also studied in \cite{Popovych, HaParkRyoo, Quinn, Giannuzzi,Basnarkov,Louca,PazoMontbrio}. Kuramoto  model is a refined model of the Winfree model . The Kuramoto model is applied for example in the Neurosciences to study the synchronization of neurones in the brain   \cite{Cumin2007,Ermentrouttheeffects}. We call natural frequency, the frequency of each oscillator, as  if it were isolated from the others.
The explicit Winfree  \cite{AriaratnamStrogatz} and Kuramoto model are defined by the following equation respectively
\begin{equation}\label{equation:WinfreeModel2}\tag{W}
\dot{x}_{i}=\omega_i+\text{Win}(X,x_i),\quad i = 1\ldots N\quad, t \geq t_0,
 \end{equation}
\begin{equation}\label{equation:KuramotoModel}\tag{K}
\dot{x}_{i}=\omega_i+\text{Kur}(X,x_i),\quad i = 1\ldots N,\quad t \geq t_0,
 \end{equation}
where for $(\omega,\kappa)\in \mathbb{R}^2_+$, $\text{Win}(Y,z)=\omega-\kappa  \frac{1}{N}\sum_{j=1}^{N}[1+\cos(y_j)]\sin(z)$  and $\text{Kur}(Y,z)=\omega-\kappa  \frac{1}{N}\sum_{j=1}^{N}\sin(y_j-z)$ for all $Y=(y_1,\ldots,y_N) \in \mathbb{R}^N$ and $z\in \mathbb{R}$. $X(t)=(x_1(t), \ldots,x_N(t))$ is the state of the systems, and $x_i(t)$ is the  phase of the i$^{\text{th}}$-oscillator. The parameter $\kappa \geq0$ is the strong coupling; the vector $(\omega_1+\omega,\ldots,\omega_N+\omega)\in \mathbb{R}^N$  is the vector of the natural frequencies. In order to apply the main result {\bf{I}} and {\bf{II}}  we need only shows that the functions  \eqref{equation:WinfreeModel2} and \eqref{equation:KuramotoModel} satisfies the synchronization hypothesis $(H_*)$ and the hypothesis  $(H)$ as proved in the following proposition.
\begin{proposition}\label{Proposition:WinKur}
There exists an open set of parameters $(\kappa,\omega) \in \mathbb{R}^2_+$,  such that the functions $\text{Win}$ and $\text{Kur}$ of the systems \eqref{equation:WinfreeModel2} and \eqref{equation:KuramotoModel} respectively, satisfies both hypotheses $(H)$ and $(H^*)$.
\end{proposition}
\begin{proof}
The function $\text{Win}$ is $C^2$ and $2\pi\mathbb{1}$-periodic. Further
\[
\min_{s\in[0,2\pi]}\text{Win}(s\mathbb{1},s)>0 \iff \forall \omega>(1+\cos(\frac{\pi}{3}))\sin(\frac{\pi}{3})\kappa,\quad\forall s \in [0,2\pi].
\]
For every $\omega>(1+\cos(\frac{\pi}{3}))\sin(\frac{\pi}{3})\kappa$ we have
\begin{align*}
\int_{0}^{2\pi}\frac{\partial_{N+1} \text{Win}(s\mathbb{1},s)}{\text{Win}(s\mathbb{1},s)}ds&=-\int_{0}^{2\pi}\frac{\kappa [1+\cos(s)]\cos(s)}{\omega-\kappa(1+\cos(s))\sin(s)}ds\\
&=-\int_{0}^{2\pi}\frac{\kappa\sin^2(s)}{\omega-\kappa(1+\cos(s))\sin(s)}ds<0.
\end{align*}
Same for the Kuramoto model, we have $\text{Kur}$ is $2\pi\mathbb{1}$-periodic, and
\[
\min_{s\in[0,2\pi]}\text{Kur}(s\mathbb{1},s)>0,\quad \forall \omega>0,\quad\forall s \in [0,2\pi],
\]
For every $\omega>0$ and $\kappa >0$ we have 
\[
\int_{0}^{2\pi}\frac{\partial_{N+1} \text{Kur}(s\mathbb{1},s)}{\text{Kur}(s\mathbb{1},s)}ds=-\int_{0}^{2\pi}\frac{\kappa}{\omega}ds=-\frac{2\pi \kappa}{\omega}<0.
\]
\end{proof}
\end{exempl}
\section{Dispersion curve}\label{Chap2dispersion_curve}
\label{pagefonctdipsrref}
The strategy to prove the mains results is to use the comparison theorem of differentials equations. We assumed a priori that the distance between the oscillators is small and find some  differential equation estimation to deduce that the distance between oscillators is bounded uniformly on time.
 We call the ``upper-solution'' {\it {the dispersion curve}}. We have the following lemma
\begin{lemma}\label{L2}
Let $\eta=(\eta_1,\eta_2,\eta_3)\in\mathbb{R}_+^3/\{(0,0,0)\}$. Let $P_1(a,b)=\eta_1 a+\eta_2 b^2$ a polynomial defined for  all $(a,b)\in \mathbb{R}\times\mathbb{R}$ and let $\Lambda: \mathbb{R}\to\mathbb{R}$  a $C^1$ and $1$-periodic function satisfying
\[
\int_{0}^1{\Lambda(s)}ds <0.
\]
Then for all $(a,b)\in\mathbb{R}_+^*\times(0,\eta_3)$ the following differential equation
\begin{equation}\label{Lin-lemme}
\frac{d}{ds}z(s)=\frac{P_1(a,b)}{\eta_3-b}+\Lambda(s)z(s),
\end{equation}
admits a positive solution $C^1$ and $1$-periodic solution that we note $\Delta_{a,b}(s)$. Further, there exists $D_{\eta,\Delta} \in (0,\eta_3)$ such that for all   $D \in (0,D_{\eta,\Delta}]$ there exists $\mathpzc{r}>0$  such that the solution $\Delta_\mathpzc{r}:=\Delta_{\mathpzc{r},D}$  satisfies
\[
\max_{s\in [0,1]}\Delta_{\mathpzc{r}}(s) \le D.
\]
\end{lemma}
\begin{proof}
Remark that for all $(a,b)\in\mathbb{R}_+^*\times(0,\eta_3)$, the differential equation \eqref{Lin-lemme} admit a  positive $C^{2}$ and $1$-periodic solution  $\Delta_{a,b}(s)$ of the form
\begin{gather*}
\Delta_{a,b}(s)=\frac{P_1(a,b)}{\eta_3-b}\frac{\int_{s}^{1+s}\exp \Big( \int_{t}^{1+s}\Lambda(v)dv \Big) dt}{1-\exp \Big( \int_{0}^{1}\Lambda(v)dv \Big)}
\end{gather*}
Put
\[
\lambda_1 =-\int_{0}^1{\Lambda(s)}ds\quad\text{and}\quad\lambda_2=\max_{0\le s,t\le 1}\int_{t}^{1+s}\Lambda(v)dv.
\]
\[
\max_{s\in[0,1]}\Delta_{a,b}(s)  \le \frac{P_1(a,b)}{\eta_3-b} \frac{\exp(\lambda_2)}{1-\exp(-\lambda_1)}. 
\]
To get  $\max_{s\in [0,1]}\Delta_{\mathpzc{r}}(s) \le D$ it sufficient to choose $\mathpzc{r}$ and $D$ such that
\begin{align}
\label{equation:figure}
\frac{P_1(\mathpzc{r},D)}{\eta_3-D}\frac{\exp(\lambda_2)}{1-\exp(-\lambda_1)}= D,
\end{align}
which is satisfied for all  $D \in (0,D_{\eta,\Lambda}]$ such that
\[
D_{\eta,\Lambda}=\frac{\eta_3}{2}\frac{1-\exp(-\lambda_1)}{1-\exp(-\lambda_1)+\eta_2\exp(\lambda_2)}.
\]
where $\mathpzc{r}>0$  is given by the following formula
\begin{equation*}
\mathpzc{r}=\frac{D}{\eta_1}\Big{[} \eta_3\frac{1-\exp(-\lambda_1)}{\exp(\lambda_2)}-[\frac{1-\exp(-\lambda_1)}{\exp(\lambda_2)}+\eta_2]D\Big{]}.
\end{equation*}
\end{proof}
\begin{definition}\label{definition:dispersion_curve}
Let $D\in (0,D_{\eta,\Lambda}]$. We call {\it  the dispersion curve associated to $D$} the solution
\[\Delta_{\mathpzc{r}}:=\Delta_{\mathpzc{r},D}(s),\] 
of the differential equation \eqref{Lin-lemme} where $\mathpzc{r}$ is defined by 
\begin{equation}
\label{radius}\mathpzc{r}=\frac{D}{\eta_1}\Big{[} \eta_3\frac{1-\exp(-\lambda_1)}{\exp(\lambda_2)}-[\frac{1-\exp(-\lambda_1)}{\exp(\lambda_2)}+\eta_2]D\Big{]},
\end{equation}
and where
\[
\lambda_1 =-\int_{0}^1{\Lambda(s)}ds\quad\text{et}\quad\lambda_2=\max_{0\le s,t\le 1}\int_{t}^{1+s}\Lambda(v)dv.
\]
\end{definition}
\begin{definition}\label{generaldefinition:Cr}
Let $D \in (0,D_{\eta,\Lambda}]$. We call {\it { the synchronization open set associated to $D$}} and we note $C_\mathpzc{r}$ the open set on $\mathbb{R}^N$ defined by
\begin{equation}
C_{\mathpzc{r}} := \Big\{ X=(x_i)_{i=1}^N \in \mathbb{R}^N \,:\, \exists \nu_X \in \mathbb{R}, \quad\max_{i}|x_i-\nu_X| < \Delta_{\mathpzc{r}}(\nu_X) \Big\},
\end{equation}
where $\Delta_{\mathpzc{r}}$ is the dispersion curve associated to $D$.
\end{definition}
\begin{remark}\label{Chap2remark:radius}
Remark that
\[
D<D_{\eta,\Lambda}<\frac{\eta_3}{2\eta_2}\exp(-\lambda_2)\quad\text{and}\quad\mathpzc{r} < D\frac{\eta_3}{\eta_1}\exp(-\lambda_2).
\]
\end{remark}
\section{Reduction of the system $(P)$}\label{Chap2synch}
The goal of this Section is to prove that the perturbed system \eqref{SystemGeneralMeanField} in particular the periodic not-perturbed system \eqref{Chap1NotPerturbedSystem}  can be studied by using a scalar periodic differential equation such as equation \eqref{Lin-lemme} of lemma \ref{L2}. 
Define the following  new system
\label{abriviatSNP}
\begin{definition}\label{SPN}
Let $X\in \mathbb{R}^N$ and let $\mu_0 \in \mathbb{R}$,  we call the (NPS) {\it{system associated to}} $\Phi^t(X)$  the not-perturbed following system
\begin{align}\label{Chap2position}\tag{NPS}
\dot{\mu}_X=F(\Phi^t(X),\mu_X),\quad t\in I_X,
\end{align}
where $I_X=[t_{0},T_X)$ is the maximal interval of the solution $X(t):=\Phi^t(X)$ of the system \eqref{SystemGeneralMeanField}  of initial condition $\phi^{t_{0}}(X)=X$. We say that $\mu_X(t)$ is the solution of the system \eqref{Chap2position} associated to $\Phi^t(X)$  of initial condition $\mu_X(t_{0})\in\mathbb{R}$.
\end{definition}
We note 
\begin{align}
\label{L}L&:=||F||_{B}+||dF||_{B}+||d^2 F||_{B},\quad \text{and}\quad \alpha&:=\min_{s\in [0,1]}F(s\mathbb{1},s).
\end{align}
Let $X\in \mathbb{R}^N$ and let $\mu_X(t)$  be the solution of the system \eqref{Chap2position} associated to $\Phi^t(X)$  of initial condition $\mu_0\in\mathbb{R}$.   We also note  $X:=\Phi^t(X)$ and $\mu_X:=\mu_X(t)$ without loss of generality. We consider the following quantities
\begin{align*}
\delta_{i,1}(X) := x_i-\mu_X,\quad\delta_{i,2}(X) := \mu_X-x_i,\\
\quad\text{and}\quad \delta(X) := \max_{1\le i\le N}|\delta_{i,1}(X)|= \max_{1\le i\le N}|\delta_{i,2}(X)|.
\end{align*}
We have the next lemma
\begin{proposition}\label{Chap2prop2}
We consider the system \eqref{SystemGeneralMeanField}. Suppose that the function $F$ satisfies the hypothesis $(H)$ and suppose that $\Phi^t(X)$  is defined for all $t\in [t_1,t_2]$. Let $D \in (0,\frac{\alpha}{L})$ and suppose that  $\delta(X)< D$ for all $t \in [t_1,t_2]$,  then
\[
\dot{\mu}_X>-LD+\alpha>0,\quad\forall t \in [t_1,t_2].
\]
In particular, $t\to \mu_X(t)$ is a diffeomorphism  from $[t_1,t_2]$ to $[\mu_X(t_1),\mu_X(t_2)]$.
\end{proposition}
\begin{proof}
The strategy is to use the Mean value theorem . Since $\delta(X)<D$  we get, $|F(X,\mu_X)-F(\mu_X\mathbb{1},\mu_X)|\le||dF||_{B}D < L D$.  Hence
\[\dot{\mu}_X=F(X,\mu_X)=[F(X,\mu_X)-F(\mu_X\mathbb{1},\mu_X)]+F(\mu_X\mathbb{1},\mu_X)>- LD+ \alpha. \]
Thanks to  hypothesis $0<D<\frac{\alpha}{L}$ to get $\dot{\mu}_X(t)>-LD+\alpha>0$  for all $t\in[t_1,t_2]$.
\end{proof}
\begin{proposition}\label{prop1}
We consider the system \eqref{SystemGeneralMeanField}. Suuppose that $F$ satisfies the hypothesis $(H)$ and suppose that $\Phi^t(X)$ is defined for all $t\in [t_1,t_2]$. Let $D \in (0,\frac{\alpha}{L})$ and  $\mathpzc{r}>0$. Suppose that
\[||H||_{B}<\mathpzc{r},\quad\text{and}\quad\delta(X)< D,\quad \forall t\in[t_1,t_2].\]
Then for all $1\le i \le N$, $k\in \{1,2\}$ and $s\in[\mu_X(t_1),\mu_X(t_2)]$ we have
\begin{equation}
\frac{d}{ds}\delta_{i,k}^*(s)<\frac{1}{\alpha}\frac{\alpha\mathpzc{r}+LD^2(L+2\alpha)}{\alpha-LD}+\frac{\partial F_{N+1}(s\mathbb{1},s)}{F(s\mathbb{1},s)}\delta_{i,k}^*(s),
\end{equation}
where\\
$\delta_{i,k}^*(s):=\delta_{i,k}(X(\mu_X^{-1}(s)))$ and $X(\mu_X^{-1}(s))=(x_1(\mu_X^{-1}(s)),\ldots,x_N(\mu_X^{-1}(s)))$.
\end{proposition}
\begin{proof}
The strategy is to use several times the Taylor formula. Let $D\in (0,\frac{\alpha}{L})$ and suppose that  $\delta(X)< D$ for all $t\in[t_1,t_2]$. Use the  Taylor formula, there exists $c_i \in [x_i,\mu_X]$  such that for all  $1 \le i \le N$
\begin{align*}
F(X,x_i)-F(X,\mu_X)&=\partial_{N+1} F(X,\mu_X)\delta_{i,1}+\frac{1}{2}\partial_{N+1}[\partial_{N+1}F(X,c_i)]\delta_{i,1}^{2}\\
&<\partial_{N+1} F(X,\mu_X)\delta_{i,1}+\frac{1}{2}||d\partial_{N+1}F||_{B}D^2\\
&<\partial_{N+1} F(X,\mu_X)\delta_{i,1}+LD^2.
\end{align*}
For $ k= 2$ we also obtain
\begin{align*}
F(X,\mu_X)-F(X,x_i)&=-\partial_{N+1} F(X,\mu_X)\delta_{i,1}-\frac{1}{2}\partial_{N+1}[\partial_{N+1}F(X,c_i)]\delta_{i,1}^{2}\\
&=\partial_{N+1} F(X,\mu_X)\delta_{i,2}-\frac{1}{2}\partial_{N+1}[\partial_{N+1}F(X,c_i)]\delta_{i,2}^{2}\\
&<\partial_{N+1} F(X,\mu_X)\delta_{i,2}+LD^2.
\end{align*}
We have $||H||_B<\mathpzc{r}$, use equations \eqref{SystemGeneralMeanField} and \eqref{Chap2position} we obtain for all $1 \le i \le N$ and $k \in \{1,2\}$ 
\begin{align}\label{Taylor:1}
\frac{d}{dt}\delta_{i,k} ={H}_i(X)+[F(X,x_i)-F(X,\mu_X)]<\mathpzc{r}+ L D^2+\partial_{N+1}F(X,\mu_X)\delta_{i,k}.
\end{align}
Use again the Taylor formula to get
\begin{align*}
\partial_{N+1}F(X,\mu_X)\delta_{i,k}&=[\partial_{N+1}F(X,\mu_X)-\partial_{N+1}F(\mu_X\mathbb{1},\mu_X)+\partial_{N+1}F(\mu_X\mathbb{1},\mu_X)]\delta_{i,k}\\
&<||d\partial_{N+1}F||_{B}D|\delta_{i,j}|+\partial_{N+1}F(\mu_X\mathbb{1},\mu_X)\delta_{i,k}\\
&< L D^2+\partial_{N+1}F(\mu_X\mathbb{1},\mu_X)\delta_{i,k}.
\end{align*}
Equation \eqref{Taylor:1}  implies that for all $1 \le i \le N$ and $k \in \{1,2\}$
\[
\frac{d}{dt}\delta_{i,k} <\mathpzc{r}+ 2 L D^2+\partial_{N+1}F(\mu_X\mathbb{1},\mu_X)\delta_{i,k}.
\]
Thanks to proposition \ref{Chap2prop2}, $\dot{\mu}_X>\alpha-LD$. We consider the change of variable : $t \to s := \mu_X(t)$ for $t\in [t_1,t_2]$. Put $\delta_{i,k}^*(s):=\delta_{i,k}(X(\mu_X^{-1}(s)))$ and $X(\mu_X^{-1}(s))=(x_1(\mu_X^{-1}(s)),\ldots,x_N(\mu_X^{-1}(s)))$. We deduce that for all $s \in [\mu_X(t_1),\mu_X(t_2)]$ 
\begin{align*}
\frac{d}{dt}\delta_{i,k}(X) &= \frac{d}{ds}\delta_{i,k}^{*}(s)\frac{d}{dt}{\mu_X}(t)<\mathpzc{r}+2L D^2+\partial_{N+1}F(s\mathbb{1},s)\delta_{i,k}^{*}(s)\\
\frac{d}{ds}\delta_{i,k}^{*}(s)&=\frac{\mathpzc{r}+2LD^2}{\dot{\mu}_X}+\frac{\partial_{N+1}F(s\mathbb{1},s)}{\dot{\mu}_X}\delta_{i,k}^{*}(s)<\frac{\mathpzc{r}+2LD^2}{\alpha-L D}+\frac{\partial_{N+1}F(s\mathbb{1},s)}{\dot{\mu}_X}\delta_{i,k}^{*}(s)\\
&=\frac{\mathpzc{r}+2L D^2}{\alpha-L D}+\frac{\partial_{N+1}F(s\mathbb{1},s)}{F(s\mathbb{1},s)}\frac{F(s\mathbb{1},s)}{\dot{\mu}_X}\delta_{i,k}^{*}(s).
\end{align*}
Use the Mean value theorem and the change of variable  $t\to s:=\mu_X(t)$ we get
\[
|F(\mu_X\mathbb{1},\mu_X)-\dot{\mu}_X|=|F(\mu_X\mathbb{1},\mu_X)-F(X,\mu_X)|<||d F||_{B}D<L D,
\]
which is equivalent to
\[
\frac{F(s\mathbb{1},s)}{\dot{\mu}_X}=1+\theta(s),\quad |\theta(s)|<\frac{LD}{\alpha-L D},\quad\forall s\in [\mu_X(t_1),\mu_X(t_2)].
\]
Finlay, since $|\frac{\partial_{N+1}F(s\mathbb{1},s)}{F(s\mathbb{1},s)}|<\frac{L}{\alpha}$ and since $|\delta_{i,k}(t)| \le \delta(X)<D$ for all $t\in[t_1,t_2]$ we obtain for all $s\in[\mu_X(t_1),\mu_X(t_2)]$
\begin{align*}
\frac{d}{ds}\delta_{i,k}^{*}(s) &<\frac{\mathpzc{r}+2LD^2}{\alpha-LD}+\frac{\partial_{N+1}F(s\mathbb{1},s)}{F(s\mathbb{1},s)}[1+\theta(s)]\delta_{i,k}^{*}(s)\\
&<\frac{\mathpzc{r}+2LD^2}{ \alpha-LD}+\frac{L^2}{\alpha}\frac{D^2}{\alpha-LD}+\frac{\partial_{N+1} F(s\mathbb{1},s)}{F(s\mathbb{1},s)}\delta_{i,k}^{*}(s)\\
&=\frac{1}{\alpha}\frac{\alpha(\mathpzc{r}+2LD^2)+L^2D^2}{\alpha-LD}+\frac{\partial_{N+1}F(s\mathbb{1},s)}{F(s\mathbb{1},s)}\delta_{i,k}^{*}(s).
\end{align*}
\end{proof}
We have the following proposition
 \begin{proposition}\label{proposition:synchintervalemaximal}
Let  $F$ be a  function  satisfying hypotheses $(H)$ and $(H_*)$. Then there exists $D_*\in(0,1)$ such that for all $D \in (0,D_*]$, there exists $\mathpzc{r}>0$   and  an open set $C_{\mathpzc{r}}$ (as in   definition \ref{generaldefinition:Cr}), such that for any function $H$ satisfying $||H||_{B}<\mathpzc{r}$ we have
\[
\forall X\in C_{\mathpzc{r}}\ :\ \Phi^{t}(X)\in C_{\mathpzc{r}},\quad\forall t\in I_X.
\]
\end{proposition}
\begin{proof}\label{reftechniquepreuvesynch}
Use equation \eqref{L} and hypothesis $(H)$
\[\max\{\int_{t_{1}}^{t_{2}} \frac{\partial_{N+1}F(s\mathbb{1},s)}{F(s\mathbb{1},s)}ds :\quad 0\le t_1\le t_2\le 1 \} \le  \frac{L}{\alpha}.\]
Let $D_{\eta,\Lambda}$ the constant defined by lemma \ref{L2} such that $\eta$ et $\Lambda$ are defined by
\[
\eta=(\frac{1}{L},2+\frac{L}{\alpha},\frac{\alpha}{L}),\quad\text{and}\quad\Lambda(s)=\frac{\partial_{N+1}F(s\mathbb{1},s)}{F(s\mathbb{1},s)}.
\]
Put $D_*:=D_{\eta,\Lambda}$.  Let $D\in (0,D_*]$  and the dispersion function  $\Delta_{\mathpzc{r}}$ associated to $D$ (See definition \ref{definition:dispersion_curve}). The dispersion curve $\Delta_{\mathpzc{r}}$ is solution of the periodic scalar differential equation
\[
\frac{d}{ds}\Delta_{\mathpzc{r}}(s)=\frac{1}{\alpha}\frac{\alpha\mathpzc{r}+LD^2(L+2\alpha)}{\alpha-LD}+\frac{F_{N+1}(s\mathbb{1},s)}{F(s\mathbb{1},s)}\Delta_{\mathpzc{r}}(s),
\]
and satisfies the folioing estimation
\[
\max_{s\in [0,1]}\Delta_{\mathpzc{r}}(s) \le D.
\]
Let  $C_{\mathpzc{r}}$ be  the synchronization open set associated to $D$, as defined in definition \ref{generaldefinition:Cr}.\\
For any function $H$ satisfying   $||H||_{B}<\mathpzc{r}$, where $\mathpzc{r}$ is given by formula \eqref{radius},  let $X(t)=(x_1(t),\ldots,x_N(t)):=\Phi^t(X)$ be the solution of the system  \eqref{SystemGeneralMeanField}  of initial condition $X=(x_1,\ldots,x_N)\in C_{\mathpzc{r}}$. There exists  $\nu_X \in \mathbb{R}$ such that $\max_{1\le i \le N}|x_i-\nu_X| < \Delta_{\mathpzc{r}}(\nu_X)\le D$. Let $\mu_X(t)$ be the solution of the system \eqref{Chap2position} associated to $X(t)$  of initial condition $\mu_X(t_{0})=\nu_X$, then $\delta(X)< \Delta_{\mathpzc{r}}(\mu_X(t_0))$. Let
\[
T^* := \sup \{t \in I_X : \forall\ t_{0}< s < t, \ \max_{i}|x_i(s)-\mu_X(s)| <\Delta_{\mathpzc{r}}(\mu_X(s)) \}.
\]
By continuity we have $t_0 \neq T^*$.  The proposition is proved if we shows that $T^*=\sup \{t\in I_X\}$. By contradiction, suppose that $ T^*\in I_X$. Using the change of variable $s=\mu_X(t)$ the proposition \ref{prop1} implies that for all $s\in [\nu_X,\mu_X^*:=\mu_X(T^*)]$
\[
\frac{d}{ds}\delta_{i,k}^{*}(s)< \frac{1}{\alpha}\frac{\alpha\mathpzc{r}+LD^2(L+2\alpha)}{\alpha-LD}+\frac{F_{N+1}(s\mathbb{1},s)}{F(s\mathbb{1},s)}\delta_{i,k}^*(s), \quad \forall s \in [\mu_X,\mu_X^*].
\]
Hence there exists  $1 \le i_0\le N$ and $k\in \{1,2\}$ such that $|\delta_{i_{0},k_0}^*(\mu_X^*)|=\Delta_{\mathpzc{r}}(\mu_X^*)$.  Suppose that  $\delta_{i_{0},k_0}^*(\mu_X^*)=\Delta_{\mathpzc{r}}(\mu_X^*)$ without loss of generality. We get
\begin{align*}
\frac{d}{ds}\delta_{i_{0},k_0}^*(\mu_X^*) &<\frac{1}{\alpha}\frac{\alpha\mathpzc{r}+LD^2(L+2\alpha)}{\alpha-LD}+\frac{F_{N+1}(\mu_X^*\mathbb{1},\mu_X^*)}{F(\mu_X^*\mathbb{1},\mu_X^*)}\delta_{i_0,k_0}^*(\mu_X^*) \\
&=\frac{1}{\alpha}\frac{\alpha\mathpzc{r}+LD^2(L+2\alpha)}{\alpha-LD}+\frac{F_{N+1}(\mu_X^*\mathbb{1},\mu_X^*)}{F(\mu_X^*\mathbb{1},\mu_X^*)}\Delta_{\mathpzc{r}}(\mu_X^*)=\frac{d}{ds}\Delta_{\mathpzc{r}}(\mu_X^*) .
\end{align*}
There exists  $s< \mu_X^*$ close enough to  $\mu_X^*$  such that $\delta_{i_{0},k_0}^*(s)>\Delta_{\mathpzc{r}}(s) $ or in other words there exists $t < T^*$ close enough to  $T^*$ such that $\delta_{i_{0},k_0}(t)>\Delta_{\mathpzc{r}}(\mu_X(t))$. We have obtained a contradiction. 
\end{proof}
\section{Proof of main result {\bf{I}} : Existence of solution and the synchronized state}\label{Chap2synch2}
 \begin{theorem}\label{Chap2proposition:dispersionCurve}
Let $F$ be a function satisfying the hypotheses  $(H)$ and $(H_*)$. Then there exists  $D_*\in(0,1)$  such that for all $D \in (0,D_*]$, there exists $\mathpzc{r}>0$  and an synchronization open set  $C_{\mathpzc{r}}$ (as  defined in the definition \ref{generaldefinition:Cr}), such that for any function $H$  satisfying $||H||_{B}<\mathpzc{r}$, and for all  $X\in C_{\mathpzc{r}}$ we have  $I_X=[t_{0},+\infty[$. Further $C_{\mathpzc{r}}$  is positive  $\Phi^t$-invariant  and
\[\forall X\in C_{\mathpzc{r}},\ \exists \nu_X \in \mathbb{R}\ :\  |\Phi^{t}_{i}(X)-\mu_X(t)|<D, \forall i=1,..,N,\ \forall t\geq t_{0},\]
 where $\mu_X(t)$  is the solution of the system \eqref{Chap2position}  associated to $\Phi^t(X)$ of initial condition $\mu_X(t_{0})=\nu_X$.
\end{theorem}
\begin{proof}
Thanks to proposition   \ref{proposition:synchintervalemaximal}, it is sufficient to prove that  $I_X=[t_{0},+\infty[$.  By contradiction suppose that there exists $t_{0}<t_X<+\infty$ such that the solution $X(t)$ is defined only on $I_X=[t_{0},t_X[$.  Then $\lim_{t\to t_X}||\Phi^t(t)||=+\infty$. Proposition \ref{proposition:synchintervalemaximal} implies that
\[
|\Phi^{t}_{i}(X)-\Phi^{t}_{j}(X)|<D, \forall1\le i,j\le N,\ \forall t_X> t\geq t_{0},
\]
For all  $i=1,..,N$,
\[
\alpha-L D-\mathpzc{r}<\frac{d}{dt}{x}_i<\max_{s\in[0,1]}F(s\mathbb{1},s)+L D+\mathpzc{r},\ \forall t_X> t\geq t_{0}.
\]
Hence $||\Phi^t(X)||<+\infty$ for all $t\in [t_{0},t_X]$, in particular $\lim_{t\to t_X}||\Phi^t(t)||<+\infty$. We have obtained a contradiction. 
\end{proof}
\section{Proof of main result {\bf{II}} : Periodic locked solution}\label{Orbiteperiodiquenombrederotation}
We use the fixed point theorem to prove the existence of periodic locked state as follow
\begin{lemma}\label{ChapPoincare}
Let $F$ be a function satisfying the hypotheses $(H)$and $(H_*)$. For any $\mathbb{1}$-periodic $C^1$ function $H$  satisfying $||H||_{B}<\mathpzc{r}$  let $C_{\mathpzc{r}}$ the synchronization  $\Phi^t$-invariant  open set given by theorem  \ref{Chap2proposition:dispersionCurve}. Let the set $\Sigma$ defined by
\[
\Sigma=\{X\in\mathbb{R}^N,\ \max_{i}|x_i|<\Delta_{\mathpzc{r}}(0)\}\subset C_{\mathpzc{r}}.
\]
Then there exists a  $C^1$ function $P : \Sigma \to \Sigma$ (the Poincar\'e map)  and a $C^1$ function  $\theta : \Sigma \to \mathbb{R}^+$ (the return time map) such that
\begin{gather*}
\Phi^{t_0+\theta(X)}(X) = P(X) + \mathbb{1}, \quad \mathbb{1}=(1, \cdots, 1) \in \mathbb{R}^N,\\
\frac{1}{L} < \theta(X) < \frac{2}{\alpha}.
\end{gather*}
\end{lemma}

\begin{proof}
Let  $X \in \Sigma \subset C_{\mathpzc{r}}$. Let  $\mu_X(t)$ be the solution of the system \eqref{Chap2position} associated to $X(t)$ of initial condition $\mu_X(t_0)=0$. Let $\tau_X$ be the inverse function of the function $\mu_X: =\mu_X(t)$ . By the proposition  \ref{Chap2prop2} and the theorem \ref{Chap2proposition:dispersionCurve}  we obtain
\[
\alpha- LD< \dot \mu_X(t) < L.
\]
Remark \eqref{Chap2remark:radius} in the Section \ref{Chap2dispersion_curve} shows that for all  $D<\frac{\alpha}{2L}$ we have
\[
\frac{\alpha}{2} < \dot \mu_X(t) < L.
\]
Let  $\theta(X) := \tau_X(1)-t_0$. Then  $\int_{t_0}^{\tau_X(1)} \dot \mu_X(t) dt = 1$ which implies the second estimation of lemma. Recall that $\max_{1 \le i \le N}|\Phi_{i}^{t}(X)-\mu_X(t)| < \Delta(\mu_X(t))<D$  for all $t \geq t_0$. Put $P(X) := \Phi^{t_0+\theta(X)}(X) - \mathbb{1}$, $P=(P_1,\ldots,P_N)$. Since $ \mu_X(t_0+\theta(X))= 1$ then
\[
\max_{1\le i\le N}|P_{i}(X)|=\max_{1\le i\le N}|\Phi_{i}^{t_0+\theta(X)}(X)-1| < \Delta(1)=\Delta(0).
\]
We have shown that $P$ is a map from $\Sigma$ into itself.
\end{proof}

\begin{corollary}\label{Chapfixedpoint}
The Poincar\'e map $P$ defined in lemma  \ref{ChapPoincare} admits a fixed point $X_* \in \Sigma$.
\end{corollary}
\begin{proof}
 $\bar \Sigma$ is compact and convex; $P : \bar\Sigma\to \bar\Sigma$ is continuous. By Brouwer fixed point theorem, $P$ admits a fixed point in $X_* \in \bar\Sigma$. We claim  that $X_* \not\in \partial \bar\Sigma$.  Suppose by contradiction $X_* \in \partial \bar\Sigma$.  There exists $1 \leq i_0 \leq N$, such that $|x_{i_{0},*}|=\Delta_{\mathpzc{r}}(0)$; 
Put
\[X_*(t)=\Phi^t(X_*),\quad X_*(t)=(x_{1,*}(t),\ldots,x_{N,*}(t)).\]
Let $\mu_{X_{*}}(t)$ the solution of the system \eqref{Chap2position}  associated to the solution $X_*(t)$ of initial condition $\mu_{X_{*}}(t_{0})=0$. We note
\[\delta_{i,1}(X_*(t))=x_{i,*}(t)-\mu_{X_{*}}(t),\quad\text{and}\quad\delta_{i,2}(X_*(t))=\mu_{X_{*}}(t)-x_{i,*}(t).\]
There exists $1 \le i \le N$, $k\in\{1,2\}$, and $t'>t_0$ close to  $t_0$ such that $\delta_{i_0,k}(X_*(t)) < \Delta_{\mathpzc{r}}(\mu_{X_{*}}(t))$ for all $t'>t>t_0$. By repeating this argument for every  $1 \leq i_1 \leq N$ and $k\in\{1,2\}$  satisfying the equality $\delta_{i_1,k}(X_*(t))=\Delta_{\mathpzc{r}}(\mu_{X_{*}}(t))$ we obtain for some $t^*>t_0$, $\delta(X_*(t^*)) < \Delta_{\mathpzc{r}}(\mu_{X_{*}}(t^*))$. But \ref{Chap2proposition:dispersionCurve} implies that   $\max_{1\le i\le N}|x_{i,*}(t)-\mu_{X_{*}}(t)|< \Delta_{\mathpzc{r}}(\mu_{X_{*}}(t))$ for all $t > t^*$. We have obtained a contradiction with the fact that
\begin{align*}
\max_{1\le i\le N}|\Phi_i^{t_0+\theta(X_*)}(X_*)-\mu_{X_{*}}(t_0+\theta(X_*))| &=\max_{1\le i\le N}|x_{i,*}+1-1| \\
&= \Delta_{\mathpzc{r}}(0)= \Delta_{\mathpzc{r}}(\mu_{X_{*}}(t_0+\theta(X_*))),
\end{align*}
knowing that  $\mu_{X_{*}}(t_0+\theta(X_*))=1$ and $\Phi^{t_0+\theta(X_*)}(X_*)=X_*+\mathbb{1}$.
\end{proof}
The main result  {\bf{II}}   is a consequence of the  previous corollary.  
\begin{proof}[Proof of the main result {\bf{II}}]
Corollary \ref{Chapfixedpoint} implies the existence of a fixed point $X_* \in C_{\mathpzc{r}}$ and return time $\theta_*>0$ such that
\[
\Phi^{t_{0}+\theta_*}(X_*) = X_* + \mathbb{1}.
\]
Thanks to periodicity and uniqueness of solution of differential equation, we obtain
\[
\Phi^{\theta_*+t}(X_*) = \Phi^t(X_*)+\mathbb{1}, \quad \forall t \geq t_0.
\]
Let $\Psi :\mathbb{R}^N\to\mathbb{R}^N$ the function defined by
\[
\Psi(s) := \Phi^s(X_*) - \frac{s}{\theta_*}\mathbb{1}=(\Psi_{1}(s),\cdots,\Psi_{N}(s)), \quad \forall s\geq t_0.
\] 
The theorem is proved if we show that $\Psi_i$ are $\theta_*$-periodic. We have
\begin{align*}
\Psi(s+\theta_*) = \Phi^{s + \theta_*}(X_*) -  \frac{s+\theta_*}{\theta_*}\mathbb{1} = \Phi^s(X_*) + \mathbb{1} - \frac{s+\theta_*}{\theta_*}\mathbb{1} = \Psi(s).
\end{align*}
Lemma \ref{ChapPoincare} implies that $\theta_*$  is uniformly bounded.
\end{proof}

\section{Conclusion}
We have generalized the result obtained in \cite{OukilKessiThieullen} to a class of abstract mean field models. We have proved the existence of solution and the existence of the synchronized solution under small perturbation. In addition,  we have proved the existence of periodic locked state for a periodic systems.


\end{document}